\documentclass[11pt, a4paper]{amsart}
\usepackage{amsmath, amsthm, amsfonts, amssymb, mathtools}
\usepackage{a4wide}
\usepackage[latin1]{inputenc}

\newcommand{\N}{\mathbb{N}}
\newcommand{\Z}{\mathbb{Z}}

\newcommand{\C}{\mathbb{C}}
\newcommand{\Q}{\mathbb{Q}}
\renewcommand{\H}{\mathbb{H}}
\newcommand{\SL}{\mathrm{SL}}
\newcommand{\Mp}{\mathrm{Mp}}
\DeclareMathOperator{\imag}{Im}
\DeclareMathOperator{\real}{Re}
\DeclareMathOperator{\e}{\mathfrak{e}}
\newcommand{\tr}{\mathrm{tr}}

\numberwithin{equation}{section}
\newtheorem{theorem}{Theorem}[section]
\newtheorem{lemma}[theorem]{Lemma}
\newtheorem{proposition}[theorem]{Proposition} 
	
\theoremstyle{definition} 
\newtheorem{example}[theorem]{Example}
\newtheorem{remark}[theorem]{Remark}

\date{\today}

\author{Steffen L\"obrich and Markus Schwagenscheidt}

\address{Korteweg-de Vries Institute for Mathematics, University of Amsterdam, Science Park 105-107, 1098 XG Amsterdam, The Netherlands}
\email{s.loebrich@uva.nl}

\address{ETH Z\"urich Mathematics Dept., R\"amistrasse 101, CH-8092 Z\"urich, Switzerland}
\email{mschwagen@ethz.ch}

\title[Fourier coefficients of meromorphic modular forms]{Arithmetic properties of Fourier coefficients of meromorphic modular forms}


\thanks{The second author is supported by SNF project 200021\_185014. We thank Vicen\c{t}iu Pa\c{s}ol and Wadim Zudilin for helpful remarks.}

\begin{document} 

\begin{abstract}
	We investigate integrality and divisibility properties of Fourier coefficients of meromorphic modular forms of weight $2k$ associated to positive definite integral binary quadratic forms. For example, we show that if there are no non-trivial cusp forms of weight $2k$, then the $n$-th coefficients of these meromorphic modular forms are divisible by $n^{k-1}$ for every natural number $n$. Moreover, we prove that their coefficients are non-vanishing and have either constant or alternating signs. Finally, we obtain a relation between the Fourier coefficients of meromorphic modular forms, the coefficients of the $j$-function, and the partition function.
\end{abstract}

\maketitle

\section{Introduction and statements of results}
\label{sec:introduction}

The study of divisibility properties of Fourier coefficients of modular forms has a long and rich history. For example, a classical result of Lehner \cite{lehner1,lehner2} asserts that the coefficients $c_j(n)$ of the modular $j$-invariant satisfy
\[
2^{3a+8}3^{2b+3}5^{c+1}7^d  \mid c_j\left( 2^a 3^b 5^c 7^d n\right)
\]
for all $a,b,c,d \in \N_0$ and $n \in \N$. In a similar spirit, Duke and Jenkins \cite{dukejenkins2} proved that the Fourier coefficients $c(n)$ of certain weakly holomorphic modular forms of positive even weight $k$ for $\Gamma := \SL_2(\Z)$ satisfy the divisibility condition $n^{k-1} \mid c(n)$ for every $n \in \N$. See also \cite{aas, griffin, hondakaneko, jenkinsmolnar, kohlberg} for some related results.

More recently, Broadhurst and Zudilin \cite{broadhurstzudilin} observed that the coefficients of a certain meromorphic modular form of weight $4$ and level $\Gamma_0(8)$ with poles in the upper half-plane seemed to satisfy $n \mid c(n)$ for $n \in \N$. This observation was proved shortly after by Li and Neururer \cite{lineururer}, using Borcherds' \cite{borcherds} regularized Shimura theta lift of weakly holomorphic modular forms of half-integral weight. Their method also revealed many other examples satisfying the divisibility condition $n \mid c(n)$ for $n \in \N$, such as the weight $4$ meromorphic modular form $\Delta/E_4^2$. Here $E_{2k}$ denotes the normalized Eisenstein series of weight $2k$ for $\Gamma$ with constant term $1$, and $\Delta = (E_4^3-E_6^2)/1728$ is Ramanujan's Delta function. Pa\c{s}ol and Zudilin \cite{pasolzudilin} applied the arguments used in \cite{lineururer} and found more examples of meromorphic modular forms satisfying similar divisiblity properties. They coined the notion of \emph{magnetic} modular forms, which are meromorphic modular forms of weight $2k$ with integral Fourier coefficients whose $(k-1)$-fold formal antiderivatives still have integral coefficients. In other words, the Fourier coefficiens $c(n)$ of magnetic modular forms of weight $2k$ satisfy the divisibility condition 
\begin{align}\label{eq divisibility condition}
n^{k-1} \mid c(n)
\end{align}
for $n \in \N$. Their new examples of such magnetic modular forms include the weight $4$ form $E_4\Delta/E_6^2$ and the weight $6$ form $E_6\Delta/E_4^3$.

In the present article, we investigate the divisibility properties of Fourier coefficients of meromorphic modular forms of positive even weight $2k$ which are associated to positive definite integral binary forms. In the spirit of the works \cite{broadhurstzudilin, lineururer, pasolzudilin}, we show that certain linear combinations of these meromorphic modular forms have integral Fourier coefficients and satisfy the divisibility condition~\eqref{eq divisibility condition}, that is, they are magnetic. Additionally, we study the non-vanishing and sign changes of these Fourier coefficients, and some connections to the partition function and the coefficients of the modular $j$-invariant. Let us describe our results in more detail.

Let $k \in \N$ with $k \geq 2$. Furthermore, for an integer $d \equiv 0,1 \pmod 4$ and a fundamental discriminant $D \in \Z$ with $dD < 0$ we let $\mathcal{Q}_{dD}$ be the set of positive definite integral binary quadratic forms of discriminant $dD$, and we let $\chi_D$ be the usual genus character on $\mathcal{Q}_{dD}$ as defined in \cite{grosskohnenzagier}. We consider the functions
\begin{align}\label{definition fkd}
f_{k,d,D}(z) := C_{k,d,D}\sum_{Q \in \mathcal{Q}_{dD}}\chi_D(Q)Q(z,1)^{-k},
\end{align}
with the constant\footnote{Our normalization of $f_{k,d,D}$ differs from the usual normalizations found in the literature.}
\begin{align}\label{eq normalizing constant}
C_{k,d,D} := \frac{(k-1)!}{(2\pi)^k}\cdot\begin{dcases}
|d|^{k-\frac{1}{2}},  & \text{if $(-1)^k d > 0$}, \\
\ell^{2k-1}|D|^{k-\frac{1}{2}},  & \text{if $(-1)^k d < 0$},
\end{dcases}
\end{align}
where we wrote $d = \ell^2 d_0$ with a fundamental discriminant $d_0$ and $\ell \in \N$. They are meromorphic modular forms of weight $2k$ for $\Gamma$ with poles at the CM points of discriminant $dD$, which are the points $z \in \H$ satisfying $Q(z,1) = 0$ for some $Q \in \mathcal{Q}_{dD}$. Moreover, the functions $f_{k,d,D}$ decay like cusp forms towards $i\infty$. Bengoechea \cite{bengoecheapaper} investigated the arithmetic nature of the Fourier coefficients of $f_{k,d,D}$ and showed that there exists a cusp form $g_{k,d,D} \in S_{2k}$ of weight $2k$ such that the difference $f_{k,d,D}-g_{k,d,D}$ has algebraic coefficients. We have the following slightly stronger integrality result.

\begin{theorem}\label{theorem integrality}
	Let $g_{k,d,D} \in S_{2k}$ be the unique cusp form whose first $\dim(S_{2k})$ Fourier coefficients agree with to those of $f_{k,d,D}$. Then $f_{k,d,D} - g_{k,d,D}$ has integral Fourier coefficients.
\end{theorem}

The proof of Theorem~\ref{theorem integrality} will be given in Section~\ref{section proof theorem integrality}. As an example, we consider the weight $8$ meromorphic modular form $f_{4,-3,1}$. Since there are no non-trivial cusp forms of weight $8$, Theorem~\ref{theorem integrality} tells us that the coefficients of $f_{4,-3,1}$ are integers. They can be computed using the formulas given in Section~\ref{section fourier expansions} below. On the other hand, noting that $f_{4,-3,1}$ has a fourth order pole at $z = e^{2\pi i/3}$, one can check by comparing the first few Fourier coefficients that we have the representation
\[
f_{4,-3,1} = -8\left(\frac{\Delta}{E_4}-3072 \frac{\Delta^2}{E_4^4}\right).
\] 
We give the prime factorizations of the coefficients $c(n)$ of $f_{4,-3,1}$ for $1 \leq n \leq 13$.
\begin{align*}
\begin{array}{c|r}
n & c(n) \\
\hline
1 & -2^3 \\
2 & 2^6 \cdot 3 \cdot 139 \\
3 & - 2^5 \cdot 3^7 \cdot 19^2\\
4 & 2^9 \cdot 11 \cdot 2701693\\
5 & -2^4 \cdot 3 \cdot 5^3 \cdot 281 \cdot 881 \cdot 4889\\
6 & 2^8 \cdot 3^8 \cdot 17 \cdot 47 \cdot 2241181\\
7 & -2^6 \cdot 7^3 \cdot 2719 \cdot 18970105159\\
8 & 2^{12} \cdot 3 \cdot 5^2 \cdot 1295477040593987\\
9 & - 2^3 \cdot 3^{14}\cdot 19 \cdot 182694956615167 \\
10 & 2^7 \cdot 3^2 \cdot 5^3 \cdot 295642982601336076331 \\
11 & - 2^5 \cdot 3 \cdot 7 \cdot 11^3 \cdot 4057 \cdot 20107 \cdot 181052802473957 \\
12 & 2^{11} \cdot 3^7 \cdot 43 \cdot 20724160121281042379621 \\
13 & -2^4 \cdot 13^3 \cdot 73 \cdot 79 \cdot 2551 \cdot 2280777977195403472231
\end{array}
\end{align*}
We notice that the coefficients $c(n)$ given in the table satisfy the divisibility property
\begin{align*}
n^{3} \mid c(n)
\end{align*}
for $1 \leq n \leq 13$. Moreover, the coefficients $c(n)$ with $3 \mid n$ are divisible by higher powers of $3$, namely by $3^{7\nu_3(n)}$. Before we state our main result explaining these phenomena, we need to introduce some more notation. Following \cite{grosszagier}, we define a \emph{relation} for $S_{2k}$ to be a sequence $\underline{\lambda} = (\lambda_m)$ of integers $\lambda_m \in \Z$ which are almost all zero and satisfy
\[
\sum_{m =1}^\infty \lambda_m c_g(m) = 0
\]
for every cusp form $g = \sum_{m=1}^\infty c_g(m)q^m \in S_{2k}$.
For $m \in \N$ we let $T_m$ denote the usual Hecke operator on modular forms of weight $2k$, see \eqref{eq hecke operator integral weight}. The \emph{Hecke translate} of $f_{k,d,D}$ corresponding to a relation $\underline{\lambda}$ is defined by
\[
f_{k,d,D,\underline{\lambda}}:= \sum_{m\geq 1} \lambda_m f_{k,d,D}|T_m.
\]
Note that any Hecke translate of a cusp form $g \in S_{2k}$ vanishes, see \cite{loebrichschwagenscheidt}, Lemma 6.1. In particular, Theorem~\ref{theorem integrality} asserts that $f_{k,d,D,\underline{\lambda}}$ has integral Fourier coefficients. We have the following divisibility property.

\begin{theorem}\label{theorem hecke translates}
	Every Hecke translate $f_{k,d,D,\underline{\lambda}}$ has integral Fourier coefficients $c(n)$ which satisfy the divisibility property $n^{k-1} \mid c(n)$ for every $n \geq 1$. More precisely, the following holds.
	\begin{enumerate}
		\item Let $(-1)^k d > 0$. If we write $n = n_1 n_2$ where every prime divisor of $n_1$ divides $D$, and $n_2$ is coprime to $D$, then $n_1^{2k-1}n_2^{k-1} \mid c(n)$.
		\item Let $(-1)^k d < 0$ and $d = \ell^2 d_0$ with a fundamental discriminant $d_0$ and $\ell \in \N$. If we write $n = n_1 n_2$ where every prime divisor of $n_1$ divides $d_0$, and $n_2$ is coprime to $d_0$, then $n_1^{2k-1}n_2^{k-1} \mid c(n)$.
	\end{enumerate}
\end{theorem}

The proof of Theorem~\ref{theorem hecke translates} can be found in Section~\ref{section proof theorem hecke translates}. It relies on the fact that the coefficients of $f_{k,d,D}$ can be written as divisor sums involving the coefficients of certain harmonic Maass-Poincar\'e series of half-integral weight, compare Proposition~\ref{proposition fourier expansion half-integral} below. Although we will prove this identity using some well-known relations between Kloosterman and Sali\'e sums, we would like to mention that it can be explained more conceptually from the fact that $f_{k,d,D}$ lies in the image of a Shimura-type theta lift, see \cite{bringmannkanevonpippich}. In particular, our method to prove Theorem~\ref{theorem hecke translates} is analogous to the one used in \cite{lineururer, pasolzudilin}.

\begin{remark}
	\begin{enumerate}
		\item For $k \in \{2,3,4,5,7\}$ we have $S_{2k} = \{0\}$, so we can choose the relation $\underline{\lambda} = (1,0,0,\dots)$. Then Theorem~\ref{theorem hecke translates} states that in these cases the $n$-th coefficient of $f_{k,d,D}$ is divisible by $n^{k-1}$ for every $d \equiv 0,1 \pmod 4$, every fundamental discriminant $D$ with $dD < 0$, and every $n \in \N$. The high $3$-divisibility of the coefficients of $f_{4,-3,1}$ of indices divisible by $3$ can be explained using item (1) of Theorem~\ref{theorem hecke translates} and the fact that $f_{4,-3,1} = f_{4,1,-3}$, see Lemma~\ref{lemma case change} below.
		\item Comparing the orders at the poles and the first few Fourier coefficients, we see that 
		\[
		f_{2,-3,1} = -64\frac{\Delta}{E_4^2}, \qquad f_{2,-4,1}= 108\frac{E_4\Delta}{E_6^2}, \qquad f_{3,-3,1} = 384\frac{E_6\Delta}{E_4^3}.
		\]
		The divisibility properties of these functions were studied in \cite{lineururer, pasolzudilin}. Although the supplements (1) and (2) of Theorem~\ref{theorem hecke translates} are not stated explicitly in \cite{lineururer, pasolzudilin}, they readily follow from their methods, as well.
		\item If $S_{2k}\neq \{0\}$, then the function $f_{k,d,D}-g_{k,d,D}$ from Theorem~\ref{theorem integrality} does usually not satisfy $n^{k-1} \mid c(n)$. We wonder if it is always possible to find a cusp form $g$ such that the $n$-th coefficient of $f_{k,d,D} - g$ if divisible by $n^{k-1}$ for every $n \in \N$.
	\end{enumerate}
\end{remark}

\begin{example}\label{example hecke translate}
	For $k = 6$ the space $S_{12}$ is spanned by $\Delta(z) = q-24q^2+252q^3+\dots$. If we choose the relation $\underline{\lambda} = (24,1,0,0,\dots)$ for $S_{12}$, then we have
	\[
	f_{6,d,D,\underline{\lambda}} = 24f_{6,d,D}+f_{6,d,D}|T_2 = f_{6,4d,D} + \left(24+32\left(\frac{d}{2} \right) \right)f_{6,d,D} + 2048 f_{6,d/4,D},
	\]
	where we used the action of $T_p$ on $f_{k,d,D}$ given in \eqref{eq hecke action fkd} below. For instance, for $d = -3$ and $D = 1$ we find
	\begin{align*}
	f_{6,-3,1,\underline{\lambda}}& = f_{6,-12,1}-8f_{6,-3,1} \\
	&= 15360(q+3471360q^2+1777624086780q^3+448590364266201088q^4+ \\
	&\qquad + 78199299812183544918750q^5+ 10856758910771768587996372992q^6+\dots).
	\end{align*}
	We see that the first few coefficients satisfy $n^{5} \mid c(n)$ and $3^{11\nu_3(n)} \mid c(n)$, in accordance with Theorem~\ref{theorem hecke translates}. It might be possible to get the same divisibility result without the factor $15360$ by a more detailed analysis, similarly to \cite{pasolzudilin}.
\end{example}

From the explicit example $f_{4,-3,1}$ given above it seems that its Fourier coefficients have alternating signs. Indeed, we have the following result.

\begin{proposition}\label{proposition sign changes}
	For $k \geq 2$ and all discriminants $d,D$ with $dD < 0$ and $D$ fundamental we have 
	\[
	(-1)^{k+ndD}c_{f_{k,d,D}}(n) > 0 ,
	\]
	for every $n \in\N$. In particular, none of the coefficients of $f_{k,d,D}$ vanishes.
\end{proposition}

We will prove the proposition in Section~\ref{section proof sign changes} by estimating the various terms in the Fourier expansion of $f_{k,d,D}$ given in Proposition~\ref{proposition fourier expansion bessel} below.

\begin{example}
	Since 
	\[
	\frac{\Delta}{E_4^2} = q - 504q^2 + 180252q^3 - 56364992q^4 + 16415391870q^5 - 4574618335008q^6 + \dots
	\] 
	is a multiple of $f_{2,-3,1}$, we see that it has non-vanishing integral Fourier coefficients with alternating signs. Similarly, 
	\[
	\frac{E_4\Delta}{E_6^2} = q + 1224q^2 + 1009692q^3 + 731063872q^4 + 493402005630q^5 + 318792567653856q^6 \dots
	\]
	is a multiple of $f_{2,-4,1}$ and hence has positive coefficients.
\end{example}

Finally, we exhibit a connection to the partition function $p(n)$ and the coefficients of the modular $j$-invariant. To this end, we need to introduce higher-level versions of the meromorphic modular forms $f_{k,d,D}$. Let $N \in \N$, let $d \equiv r^2 \pmod{4N}$ be a discriminant and $D \equiv \rho^2 \pmod{4N}$ a fundamental discriminant with $dD < 0$, for some $r,\rho \in \Z$. Then we define the function
\[
f_{k,d,r,D,\rho}(z) := C_{k,d,D}\sum_{Q \in \mathcal{Q}_{N,dD,r\rho}}\chi_D(Q)Q(z,1)^{-1},
\]
where $\mathcal{Q}_{N,dD,r\rho}$ is the set of all positive definite integral binary quadratic forms $[a,b,c]$ of discriminant $dD$ with $N \mid a$ and $b \equiv r\rho \pmod{2N}$, and $\chi_D$ is the generalized genus character on $\mathcal{Q}_{N,dD,r\rho}$ as defined in \cite{grosskohnenzagier}. The constant $C_{k,d,D}$ is defined as in \eqref{eq normalizing constant}. We would like to remark that the magnetic modular form $\phi$ of weight $4$ for $\Gamma_0(8)$ considered in \cite{broadhurstzudilin,lineururer} is a multiple of the meromorphic modular form $f_{2,-16,4,1,1}-f_{2,-16,-4,1,1}$ of level $N = 8$, which can be checked by comparing the orders at the poles and the first few Fourier coefficients.

We focus on a particular example in level $N = 6$ and weight $2k = 4$, where we let $D =\rho=1$ and omit them from the notation for brevity. Furthermore, we choose $d = -23$, and consider the linear combination
\begin{align}\label{eq Fminus23}
F_{-23} := 
\sum_{r(12)}\left(\frac{12}{r}\right)f_{2,-23,r} =
 f_{2,-23,1}-f_{2,-23,5}-f_{2,-23,7}+ f_{2,-23,11}.
\end{align}
This is a meromorphic modular form of weight $4$ for $\Gamma_0(6)$ that vanishes at all cusps. Its Fourier coefficients $c_{F_{-23}}(n)$ are related to the partition function $p(n)$ and the coefficients $c_j(n)$ of the $j$-invariant by the following formulas.

\begin{proposition}\label{proposition combinatorial interpretation}
	For $n \geq 1$ we have
	\begin{align*}
	c_{F_{-23}}(n) &= n\sum_{m\mid n}\left(\frac{-23}{n/m} \right)\left( \frac{12}{m}\right)m^2 p\left(\frac{23m^2+1}{24} \right) \\
	&= \sum_{m \mid n}\left(\frac{12}{m} \right)\frac{n}{m}\sum_{d =1}^{\lfloor \sqrt{m^2+24}\rfloor}\left(\frac{12}{d}\right)\frac{m^2-d^2}{24}c_j\left( \frac{m^2-d^2}{24}\right).
	\end{align*}
\end{proposition}

The proof will be given in Section~\ref{section proof combinatorial interpretation}. The basic idea for the first identity is that $F_{-23}$ can be written as a Shimura-type theta lift of the inverse $\eta^{-1}(\tau)$ of the Dedekind eta function. Note that $\eta^{-1}(\tau)$ is a weakly holomorphic modular form of weight $-\frac{1}{2}$ whose coefficients are given by the partition function. Similarly, the second identity follows from the fact that $F_{-23}$ can also be written as a theta lift of the weight $\frac{5}{2}$ weakly holomorphic modular form $\eta(\tau)j'(\tau)$, where $j'(\tau)= \frac{1}{2\pi i}\frac{\partial}{\partial \tau}j(\tau)$. However, for brevity, we will not use the theta lift machinery, but show the above identities directly using identities between Kloosterman and Sali\'e sums.

As a corollary to Proposition~\ref{proposition combinatorial interpretation} we obtain identities between the partition function and the coefficients of the modular $j$-invariant, which are of independent interest. For some explicit instances, we let $n = 5$ and $n = 7$, and find
\begin{align*}
p(24) &= \frac{1}{125}\left(-10+c_j(1)+c_j(-1) \right) = \frac{1}{125}(-10+196884+1) = 1575,  \\
p(47) &= \frac{1}{343}\left(-14+2c_j(2)-c_j(1) \right) = \frac{1}{343}(-14+2\cdot 21493760-196884) = 124754.
\end{align*}
It would be interesting to find a more direct proof for the identities implied by Proposition~\ref{proposition combinatorial interpretation}.

Finally, we would like to remark that the occurrence of the Kronecker symbol $\left(\frac{12}{m} \right)$ in the above formulas implies that $c_{F_{-23}}(n) = n$ whenever $n$ has only the prime divisors $2$ and $3$. This gives a quite interesting looking Fourier expansion. To give an impression, we list the first $20$ coefficients of $F_{-23}$.
\begin{align*}
\begin{array}{r|r}
n & c_{F_{-23}}(n) \\
\hline
1 & 1\\
2 & 2\\
3 & 3\\
4 & 4\\
5 & -196880\\
6 & 6 \\
7 & -42790629\\
8 & 8 \\
9 & 9 \\
10 & -393760 
\end{array} \qquad\qquad 
\begin{array}{r|r}
n & c_{F_{-23}}(n) \\
\hline
11 & 1582436878077\\
12 & 12\\
13 & 285420848487502\\
14 & -85581258\\
15 & -590640\\
16 & 16\\
17 & -8658073941610362614\\
18 & 18\\
19 & -1472066917939200724860 \\
 20 &  -787520 
\end{array}
\end{align*}

The work is organized as follows. We start with a preliminary section about harmonic Maass forms of half-integral weight, Hecke operators, and some properties of the meromorphic modular forms $f_{k,d,D}$. In the remaining sections, we give the proofs of our four main results from the introduction.

\section{Preliminaries}

\subsection{Harmonic Maass forms of half-integral weight}
\label{section harmonic maass forms}

We briefly recall some basic facts about harmonic Maass forms from \cite{bruinierfunke}. Let $\kappa \in \Z$. A smooth function $f: \H \to \C$ is called a harmonic Maass form of weight $\frac{1}{2}+\kappa$ if it transforms like a modular form of weight $\frac{1}{2}+\kappa$ for $\Gamma_0(4)$ (in the usual sense as in \cite{kohnenhalfintegral}), if it is annihilated by the invariant Laplace operator
\begin{align}\label{eq laplace}
\Delta_{\frac{1}{2}+\kappa} = -v^2 \left(\frac{\partial^2}{\partial u^2} + \frac{\partial^2}{\partial v^2} \right) + iv\left(\frac{1}{2}+\kappa\right)\left(\frac{\partial}{\partial u}+i\frac{\partial}{\partial v} \right), \qquad (\tau = u+iv \in \H),
\end{align}
and if it is mapped to a cusp form of weight $\frac{3}{2}-\kappa$ under the antilinear differential operator
\[
\xi_{\frac{1}{2}+\kappa} = 2iv^{\frac{1}{2}+\kappa}\overline{\frac{\partial}{\partial \overline{\tau}}}.
\] 
We let $H_{\frac{1}{2}+\kappa}$ be the space of harmonic Maass forms of weight $\frac{1}{2}+\kappa$ satisfying the Kohnen plus space condition, which means that $F \in H_{\frac{1}{2}+\kappa}$ has a Fourier expansion of the shape
\[
F(\tau) = \sum_{\substack{(-1)^{\kappa}n \equiv 0,1 (4) \\ n\gg -\infty}}c_F^+(n)q^{n} + \sum_{\substack{(-1)^{\kappa}n \equiv 0,1(4) \\ n< 0}}c_F^-(n)\Gamma\left(\frac{1}{2}-\kappa,4\pi |n| v\right)q^{n},
\]
with $c_F^{\pm}(n) \in \C$, where $\Gamma(s,x) = \int_x^\infty e^{-t}t^{s-1}dt$ denotes the incomplete gamma function. The series involving the coefficients $c_F^+(n)$ (resp. $c_F^-(n)$) is called the holomorphic (resp. non-holomorphic) part of $F$. We let $M_{\frac{1}{2}+\kappa}^!$ be the subspace of weakly holomorphic modular forms, which are holomorphic on $\H$, and we let $S_{\frac{1}{2}+\kappa}$ be the subspace of cusp forms.

Important examples of harmonic Maass forms are given by Maass-Poincar\'e series, see \cite{bringmannono}. For every $n > 0$ with $(-1)^{\kappa+1} n \equiv 0,1 \pmod 4$ one can construct a Maass-Poincar\'e series $P_{\frac{1}{2}+\kappa,n} \in H_{\frac{1}{2}+\kappa}$ which satisfies
\[
P_{\frac{1}{2}+\kappa,n}(\tau) = q^{-n}+O(1).
\]
If $\kappa \geq 2$ then $P_{\frac{1}{2}+\kappa,n}$ is weakly holomorphic since there are no cusp forms of negative weight. 

The following lemma due to Duke and Jenkins \cite{dukejenkins} will be crucial for our integrality and divisibility results.

\begin{lemma}\label{lemma integral basis}
		Let $k \in \Z$ with $k \geq 2$.
		\begin{enumerate}
			\item If $F(\tau) = \sum_{n \gg -\infty}c_F(n)q^n \in M_{\frac{3}{2}-k}^!$ has integral coefficients $c_F(n)$ for $n < 0$, then all coefficients of $F$ are integers.
			\item If $F(\tau) = \sum_{n \gg -\infty}c_F(n)q^n \in M_{\frac{1}{2}+k}^!$ has integral coefficients $c_F(n)$ for $n < 0$, then there exists a cusp form $G \in S_{\frac{1}{2}+k}$ such that all of the coefficients of $F-G$ are integers.
		\end{enumerate}
\end{lemma}

\begin{proof}
	We only show the first claim, since the second one can be proved analogously. Let $A$ denote the maximal order of vanishing at $\infty$ of elements in $M_{\frac{3}{2}-k}^!$. Note that $A < 0$ since $k \geq 2$. In \cite{dukejenkins}, the authors construct a basis of $M_{\frac{3}{2}-k}^!$ of the form
	\[
	F_{k,m}(\tau) = q^{-m} + \sum_{n > A}c_k(m,n)q^n, \qquad (m \geq -A),
	\]
	with integral coefficients $c_k(m,n) \in \Z$. We have 
	\[
	F = \sum_{m \geq -A}c_F(-m)F_{k,m}
	\]
	since the difference of the two sides has vanishing order at $\infty$ bigger than $A$, and hence vanishes identically. Since $A < 0$, the above linear combination involves only the coefficients $c_F(n)$ with $n < 0$, which are integral by assumption. This implies that $F$ has integral coefficients.
\end{proof}

\subsection{Hecke operators}
\label{section hecke operators}

	Let $k \in \Z$ with $k \geq 2$ as before. For $m \in \N$ the Hecke operator $T_m$ acts on weakly holomorphic or meromorphic modular forms $f$ of weight $2k$ by
	\begin{align}\label{eq hecke operator integral weight}
	f|T_m = \sum_{n \gg -\infty}\sum_{d \mid (m,n)}d^{2k-1}c_f\left( mn/d^2\right)q^n.
	\end{align}
	The Hecke operators $T_m$ are multiplicative, and for prime power index satisfy the recursion
	\begin{align}\label{eq hecke recursion integral weight}
	T_{p^{\ell+1}} = T_{p^\ell}T_p-p^{2k-1}T_{p^{\ell-1}}.
	\end{align}
	By an analogous computation as in the proof of equation (36) in \cite{zagiereisenstein}, one can show that $T_{p}$ acts on $f_{k,d,D}$ by
	\begin{align}\label{eq hecke action fkd}
	f_{k,d,D}|T_p = f_{k,p^2 d,D} + \left( \frac{d}{p}\right)p^{k-1}f_{k,d,D} + p^{2k-1}f_{k,d/p^2,D}.
	\end{align}
	In particular, we see that $f_{k,d,D}|T_m$ is a linear combination of the functions $f_{k,n^2d,D}$ for suitable $n \in \Q$.
	
	Following \cite{kohnenhalfintegral}, for a prime $p$ we define the Hecke operator $T_{p^2}$ on weakly holomorphic modular forms $F \in M_{\frac{1}{2}+k}^!$ of positive half-integral weight by
	\[
	F|T_{p^2} = \sum_{(-1)^k n\equiv 0,1 (4)}\left(c_F(p^2 n) + \left( \frac{(-1)^k n}{p}\right)p^{k-1}c_F(n) + p^{2k-1}c_F(n/p^2) \right)q^n.
	\]
	We remark that one should more precisely denote this operator by $T_{p^2}^+$ to distinguish it from the Hecke operator on the full space of half-integral weight modular forms (not necessarily satisfying the Kohnen plus space condition). Since we are only working with the plus space, we permit ourselves this slight abuse of notation. Similarly, we define the action of $T_{p^2}$ on modular forms $F \in M_{\frac{3}{2}-k}^!$ of negative half-integral weight by
	\begin{align}\label{eq hecke action half-integral}
	F|T_{p^2} = \sum_{(-1)^k n\equiv 0,3 (4)}\left(p^{2k-1}c_F(p^2 n) + \left( \frac{(-1)^{k+1} n}{p}\right)p^{k-1}c_F(n) + c_F(n/p^2) \right)q^n.
	\end{align}
	This formula is obtained by taking the weight $\frac{1}{2}+k$ Hecke operator (allowing negative $k$ for the moment), and then renormalizing by multiplying with $p^{1-2k}$. Here we follow the convention of \cite{dukejenkins}. Note that $T_{p^2}$ acts on harmonic Maass forms in $H_{\frac{3}{2}-k}$ by an analogous formula. For both $M_{\frac{1}{2}+k}^!$ and $H_{\frac{3}{2}-k}$ we define the Hecke operator of prime power index by the recursion
	\begin{align}\label{eq hecke recursion half-integral weight}
	T_{p^{2 \ell + 2}} = T_{p^{2\ell}}T_{p^2} - p^{2k-1}T_{p^{2\ell - 2}},
	\end{align}
	and we define $T_{m^2}$ for $m \in \N$ multiplicatively. This is precisely the same recursion as \eqref{eq hecke recursion integral weight} for the integral weight $2k$ Hecke operator. 
	
	It follows from the action \eqref{eq hecke action half-integral} on the Fourier expansion that $T_{p^2}$ acts on the Maass-Poincar\'e series $P_{\frac{3}{2}-k,n}$ by
	\begin{align}\label{eq hecke action poincare series}
	P_{\frac{3}{2}-k,n}|T_{p^2} = P_{\frac{3}{2}-k,p^2n} + \left( \frac{(-1)^{k}n}{p}\right)p^{k-1}P_{\frac{3}{2}-k,n} + p^{2k-1}P_{\frac{3}{2}-k,n/p^2}.
	\end{align}
	Notice the analogy to the action \eqref{eq hecke action fkd} of $T_p$ on $f_{k,d,D}$.

\subsection{The Fourier expansion of $f_{k,d,D}$}
\label{section fourier expansions}

Let $k \in \Z$ with $k \geq 2$. We write
\[
f_{k,d,D}(z) = \sum_{n \geq 1}c_{f_{k,d,D}}(n)q^{n}
\]
for the Fourier expansion of $f_{k,d,D}$, which converges for $\imag(z) > \frac{\sqrt{|dD|}}{2}$. We start with the well-known expansion in terms of Sali\'e sums and Bessel functions, whose computation can be found in \cite{bengoecheapaper}, Proposition 2.2.

\begin{proposition}\label{proposition fourier expansion bessel}
For $n \geq 1$ we have
\[
c_{f_{k,d,D}}(n) = C_{k,d,D}\frac{2^{k+\frac{1}{2}}\pi^{k+1}}{(k-1)!}(-1)^k  |dD|^{\frac{1}{4}-\frac{k}{2}} n^{k-\frac{1}{2}}\sum_{\substack{a \geq 1}}a^{-\frac{1}{2}}S_{a,d,D}(n)I_{k-\frac{1}{2}}\left( \frac{\pi n\sqrt{|dD|}}{a}\right),
\]
with the constant $C_{k,d,D}$ defined in \eqref{eq normalizing constant}, the $I$-Bessel function $I_\nu(x)$, and the Sali\'e sum
\[
S_{a,d,D}(n) := \sum_{\substack{b (2a) \\ b^{2} \equiv dD (4a)}}\chi_D\left(\left[a,b,\frac{b^2-dD}{4a}\right]\right) e\left(\frac{nb}{2a}\right).
\]
\end{proposition}

Next, we write the Fourier coefficients of $f_{k,d,D}$ in terms of the coefficients of the Maass-Poincar\'e series of half-integral weight defined in Section~\ref{section harmonic maass forms}. 

\begin{proposition}\label{proposition fourier expansion half-integral}
	For $n \geq 1$ we have
	\begin{align*}
	c_{f_{k,d,D}}(n) &= 
	\begin{dcases} 
	(-1)^{\left\lfloor\frac{k}{2}\right\rfloor+1} n^{2k-1}\sum_{m|n} \left(\frac{D}{m}\right) m^{-k} c_{P_{\frac{3}{2}-k,|d|}}^+\left(\frac{n^{2}|D|}{m^{2}}\right), &  \text{if } (-1)^k d> 0, \\
	(-1)^{\left\lfloor\frac{k}{2}\right\rfloor} \ell^{2k-1}\sum_{m|n} \left(\frac{D}{m}\right) m^{k-1} c_{P_{\frac{1}{2}+k,|d|}}\left(\frac{n^{2}|D|}{m^{2}}\right) ,& \text{if } (-1)^k d < 0,
	\end{dcases}
	\end{align*}
	where we write $d = \ell^2 d_0$ with a fundamental discriminant $d_0$ and $\ell \in \N$, and $c_{P_{\frac{3}{2}-k,|d|}}^+(n)$ denotes the $n$-th coefficient of the holomorphic part of the harmonic Maass form $P_{\frac{3}{2}-k,|d|}$.
\end{proposition}

\begin{proof}
The formula in the first case has been proven in \cite{loebrichschwagenscheidt}, Lemma~7.2, where more details can be found.
For the second case, we can rewrite the Fourier expansion of the half-integral weight Maass-Poincar\'e series given in \cite{bringmannono}, Theorem~2.1, as
\begin{align}\label{maasspoincareexpansion}
c_{P_{\frac{1}{2}+k,m}}\left(n\right)=
(-1)^{\lfloor\frac{k+1}{2}\rfloor}\pi\sqrt{2}\left(\frac{n}{m}\right)^{\frac{k}{2}-\frac14}\sum_{a\geq 1} 
\frac{K^+((-1)^{k+1}m,(-1)^kn, a)}{a}I_{k-\frac{1}{2}}\left(\frac{\pi\sqrt{mn}}{a}\right)
\end{align}
with the half-integral weight Kloosterman sum 
\begin{align*}
K^+(m,n,a):= \frac{1-i}{4}\left(1+\left(\frac{4}{a}\right)\right)\sum_{j\text{ mod}^*(4a)}\left(\frac{4a}{j}\right)\left(\frac{-4}{j}\right)^{\frac12}e\left(\frac{mj+n\overline{j}}{4a}\right).
\end{align*}
Now we can apply the identity between Kloosterman and Sali\'e sums 
\begin{align}\label{salieidentity}
S_{a,d,D}(n) = \sum_{m|(n,a)}\left(\frac{
D}{m}\right)\sqrt{\frac{m}{a}}K^+\left(d,\frac{n^2D}{m^2},\frac{a}{m}\right) 
\end{align}
given in \cite{dukeimamoglutoth}, Proposition 3 (note that the Kloosterman and Sali\'e sums are normalized slightly differently there). Plugging \eqref{maasspoincareexpansion} and \eqref{salieidentity} into the Fourier expansion of $f_{k,d,D}$ from Proposition~\ref{proposition fourier expansion bessel} above yields the result. 
\end{proof}

Finally, we give an expansion for $f_{k,d,D}$ in terms of traces of CM values of derivatives of Maass-Poincar\'e series. For $n \in \N$ we let $P_{2-2k,n}$ be the unique harmonic Maass form of weight $2-2k$ for $\Gamma$ with $P_{2-2k,n} = q^{-n}+O(1)$. It can be constructed explicitly as a Maass Poincar\'e series as in \cite{bruinierhabil}, Section~1.3. We recall this construction in the proof of Proposition~\ref{proposition fourier expansion CM} below. Furthermore, for $\kappa \in \frac{1}{2}\Z$ we consider the differential operator
\[
\partial_\kappa := \frac{1}{2\pi i} \frac{\partial}{\partial z}-\frac{\kappa}{4\pi y},
\]
which raises the weight of a Maass form of weight $\kappa$ by $2$. Note that $\partial_\kappa = -\frac{1}{4\pi}R_{\kappa}$ with the usual Maass raising operator $R_{\kappa} := 2i\frac{\partial}{\partial z} + \frac{\kappa}{y}$. We define the iterated version (raising from weight $2-2k$ to weight $0$) by
\[
\partial^{k-1} := (-1)^{k-1}\partial_{-2} \circ \dots \circ\partial_{2-2k}.
\]
We define the trace of CM values of a $\Gamma$-invariant function $f$ by
\[
\tr_{d,D}(f) := \sum_{Q \in \mathcal{Q}_{dD}}\chi_D(Q)\frac{f(z_Q)}{w_Q},
\]
where $w_Q$ is the order of the stabilizer of $Q$ in $\overline{\Gamma} = \mathrm{PSL}_2(\Z)$. Furthermore, for a function $f$ transforming like a modular form of weight $2-2k$ for $\Gamma$ we consider the normalized\footnote{Our normalization slightly differs from the normalization in \cite{dukejenkins}.} trace
\begin{align*}
\tr_{d,D}^*(f) &:= C_{k,d,D}\frac{(2\pi)^k}{(k-1)!}(-1)^k|dD|^{-\frac{k}{2}}\tr_{d,D}(\partial^{k-1}f),
\end{align*}
with the constant $C_{k,d,D}$ defined in \eqref{eq normalizing constant}. We remark that, by the results of \cite{dukejenkins}, the traces $\tr_{d,D}^*(f)$ are integers if $f$ is a weakly holomorphic modular form with integral Fourier coefficients.

\begin{proposition}\label{proposition fourier expansion CM}
	For $n\geq 1$ we have
	\[
	c_{f_{k,d,D}}(n) = \tr_{d,D}^*\left(P_{2-2k,n}\right).
	\]
\end{proposition}

\begin{proof}
	Following \cite{bruinierhabil}, Section 1.3, for $\kappa \in \Z, s \in \C$ and $v > 0$ we define
\[
\mathcal{M}_{\kappa,s}(v) := v^{-\frac{\kappa}{2}}M_{-\frac{\kappa}{2},s-\frac{1}{2}}(v)
\]
with the usual $M$-Whittaker function. Note that for $\kappa = 0$ we have
	\[
	\mathcal{M}_{0,s}(v) = 2^{2s-1}\Gamma\left(s+\frac{1}{2}\right)v^{\frac{1}{2}}I_{s-\frac{1}{2}}\left(\frac{v}{2}\right)
	\]
	with the $I$-Bessel function. For $n > 0$ we consider the Maass Poincar\'e series
	\[
	P_{\kappa,n,s}(z) := \frac{1}{\Gamma(2s)}\sum_{M \in \Gamma_{\infty}\backslash \Gamma}\mathcal{M}_{\kappa,s}(4\pi ny)e(-nx)|_{\kappa}M, \qquad (z = x+ iy \in \H),
	\]
	where $\Gamma_\infty$ is the subgroup generated by $\pm\left(\begin{smallmatrix} 1 & 1 \\ 0 & 1 \end{smallmatrix}\right)$.	If $\kappa < 0$, then the special value $P_{\kappa,n}:=P_{\kappa,n,1-\frac{\kappa}{2}}$ defines a harmonic Maass form of weight $\kappa$ with $P_{\kappa,n} = q^{-n}+O(1)$. The Maass Poincar\'e series of different weights are related by
	\[
	\partial_\kappa P_{\kappa,n,s} = -n\left( s+\frac{\kappa}{2}\right)P_{\kappa+2,n,s},
	\]
	see \cite{bruinierono}, Proposition 2.2. Combining the above results, and applying the Legendre duplication formula $2^{2k-1}\Gamma(k)\Gamma(k+\frac{1}{2}) = \sqrt{\pi}\Gamma(2k)$, we obtain
	\[
	\partial^{k-1} P_{2-2k,n} = 2\pi n^{k-\frac{1}{2}}\sum_{M \in \Gamma_{\infty}\backslash \Gamma} y^{\frac{1}{2}}I_{k-\frac{1}{2}}(2\pi n y)e(-nx)|_{0}M.
	\]
	
	Plugging in the CM points of discriminant $dD$, we find
	\begin{align*}
	&\tr_{d,D}(\partial^{k-1}P_{2-2k,n}) \\
	&= 2\pi n^{k-\frac{1}{2}}\sum_{Q \in \mathcal{Q}_{dD}/\Gamma}\frac{\chi_D(Q)}{w_Q}\sum_{M \in \Gamma_{\infty}\backslash \Gamma} (\imag(Mz_Q))^{\frac{1}{2}}I_{k-\frac{1}{2}}(2\pi n \imag(Mz_Q))e(-n\real(Mz_Q)).
	\end{align*}
	 Note that for $Q = [a,b,c] \in \mathcal{Q}_{dD}$ the associated CM point is given by $z_Q = (-b+i\sqrt{|dD|})/2a$. Moreover, if $M$ runs through $\Gamma_\infty \backslash \Gamma$, then $Mz_Q$ runs $w_Q$ times through the CM points associated to the forms in $[Q]/\Gamma_\infty$, where $[Q]$ denotes the class of $Q$ in $\mathcal{Q}_{dD}/\Gamma$. Thus we obtain
	 \begin{align*}
	 \tr_{d,D}(\partial^{k-1}P_{2-2k,n}) = 2\pi n^{k-\frac{1}{2}}\sum_{[a,b,c] \in \mathcal{Q}_{dD}/\Gamma_\infty}\chi_D(Q) \left(\frac{\sqrt{|dD|}}{2a}\right)^{\frac{1}{2}}I_{k-\frac{1}{2}}\left(\frac{\pi n\sqrt{|dD|}}{a} \right)e\left(\frac{nb}{2a}\right).
	 \end{align*}
	 A system of representatives for $\mathcal{Q}_{dD}/\Gamma_\infty$ is given by the forms $[a,b,c]$ with $a > 0, b \pmod{2a}$ with $b^2 = dD \pmod{4a}$, and $c = (b^2-dD)/4a$. Comparing the resulting expression for $\tr_{d,D}(\partial^{k-1}P_{2-2k,n})$ with the Fourier expansion of $f_{k,d,D}$ given in Proposition~\ref{proposition fourier expansion bessel}, and taking into account the additional factors in the normalized trace $\tr_{d,D}^*(P_{2-2k,n})$, we obtain the stated formula.	 
\end{proof}

\subsection{Interchanging $d$ and $D$} If $d_0$ and $D$ are both fundamental discriminants with $d_0 D < 0$, then it follows from the fact that $\chi_D = \chi_{d_0}$ on $\mathcal{Q}_{d_0 D}$ that 
\[
f_{k,d_0,D} = f_{k,D,d_0}.
\]
For general discriminants $d$ we have the following relation.

\begin{lemma}\label{lemma case change}
	For fundamental discriminants $d_0$ and $D$ with $d_0D < 0$ and $\ell \in \N$ we have
	\[
	f_{k,\ell^2 d_0,D} = \sum_{a|\ell}\left(\frac{D}{a}\right)\sum_{b|\frac{\ell}{a}}\mu(b)\left(\frac{d_0}{b}\right)(ab)^{k-1}f_{k,\left(\frac{\ell}{ab}\right)^2 D,d_0}(z).
	\]
\end{lemma}

\begin{proof}
We let $\mathcal{Q}^0_{dD}$ be subset of primitive forms in $\mathcal{Q}_{dD}$ and define
$$
f_{k,d,D}^0(z):=C_{k,d,D}\sum_{Q \in \mathcal{Q}^0_{dD}}\chi_D(Q)Q(z,1)^{-k}.
$$ 
Then we have
\begin{align}
\begin{split}\label{sumftilde}
f_{k,\ell^2 d_0,D}(z) &= C_{k,\ell^2 d0,D}\sum_{Q\in\mathcal{Q}_{\ell^2 d_0D}}\chi_{D}(Q)Q(z,1)^{-k}\\
& =\sum_{a|\ell}a^{2k-1}C_{k,\left(\frac{\ell}{a}\right)^2 d_0,D}\sum_{Q\in\mathcal{Q}^0_{\left(\frac{\ell}{a}\right)^2d_0D}}\chi_{D}(aQ)(aQ(z,1))^{-k} \\
&=\sum_{a|\ell}a^{k-1}\left(\frac{D}{a}\right)f^0_{k,\left(\frac{\ell}{a}\right)^2d_0,D}(z)\\
&=\sum_{a|\ell}a^{k-1}\left(\frac{D}{a}\right)f^0_{k,\left(\frac{\ell}{a}\right)^2 D,d_0}(z).
\end{split}
\end{align}
The last identity follows since we have $\chi_D(Q) = \chi_{d_0}(Q)$ for primitive forms $Q$.

On the other hand, expressing $f^0_{k,d,D}$ in terms of $f_{k,d,D}$ gives 
 \begin{align*}
f^0_{k,t^2 D,d_0}(z) &= C_{k,t^2 D,d_0}\sum_{b|t}\mu(b)\sum_{Q\in\mathcal{Q}_{\left(\frac{t}{b}\right)^2 Dd_0}}\chi_{d_0}(bQ)(bQ)(z,1)^{-k}\\
& =\sum_{b|t}\mu(b)\left(\frac{d_0}{b}\right)b^{k-1}f_{k,\left(\frac{t}{b}\right)^2 D, d_0}(z).
\end{align*}
Plugging this back into \eqref{sumftilde} proves the result.
\end{proof}

\section{The proof of Theorem~\ref{theorem integrality}}
\label{section proof theorem integrality}

	Let $(-1)^k d < 0$. By Lemma~\ref{lemma integral basis}, there exists a cusp form $G \in S_{\frac{1}{2}+k}$ such that $P_{\frac{1}{2}+k,|d|}-G$ has integral coefficients. The $D$-th Shimura lift 
	\[
	\mathcal{S}_D(G)(z) := (-1)^{\left\lfloor\frac{k}{2}\right\rfloor}\sum_{n\geq 1}\sum_{m|n} \left(\frac{D}{m}\right) m^{k-1} c_{G}\left(\frac{n^{2}|D|}{m^{2}}\right)  q^n
	\] 
	is a cusp form of weight $2k$, compare \cite{kohnenhalfintegral}. From Proposition~\ref{proposition fourier expansion half-integral} we obtain that
	\[
	\ell^{1-2k}f_{k,d,D}-\mathcal{S}_D(G) = \mathcal{S}_D\left(P_{\frac{1}{2}+k,|d|}-G\right)
	\]
	has integral coefficients. Since $S_{2k}$ has a basis of the form $q^m+O(q^{\dim(S_{2k})+1})$ for $1 \leq m \leq \dim(S_{2k})$ with integral coefficients (see \cite{dukejenkins}, Section 2), we can subtract a cusp form $g$ with integral coefficients to achieve that 
	\[
	\ell^{1-2k}f_{k,d,D}-\mathcal{S}_D(G)-g
	\]
	has integral coefficients which vanish for $1 \leq n \leq \dim(S_{2k})$. Then $g_{k,d,D} = \ell^{2k-1}(\mathcal{S}_D(G)+g)$ is the desired cusp form.
	
	The case $(-1)^k d > 0$ follows from the case $(-1)^k d < 0$ using Lemma~\ref{lemma case change}, concluding the proof.

\section{The proof of Theorem~\ref{theorem hecke translates}}
\label{section proof theorem hecke translates}

We let
\[
F(\tau) = \sum_{n \gg -\infty}c_F(n)q^n \in M_{\frac{3}{2}-k}^!
\]
be a weakly holomorphic modular form of weight $\frac{3}{2}-k$ for $\Gamma_0(4)$ satisfying the Kohnen plus space condition $c_F(n) = 0$ unless $(-1)^k n \equiv 0,3 \pmod 4$. We further assume that the coefficients $c_F(n)$ of negative index $n < 0$ are integers. 

We will obtain Theorem~\ref{theorem hecke translates} as a special case of the following result.

\begin{theorem}\label{theorem halfintegral}
	For a fundamental discriminant $D$ with $(-1)^k D < 0$ the linear combination
\begin{align}\label{eq halfintegral 1}
\sum_{(-1)^k d > 0}c_F(-|d|)f_{k,d,D}
\end{align} has integral Fourier coefficients $c(n)$ which satisfy $n^{k-1}\mid c(n)$ for every $n \geq 1$. More precisely, if we write $n = n_1 n_2$ where every prime divisor of $n_1$ divides $D$, and $n_2$ is coprime to $D$, then the coefficients of \eqref{eq halfintegral 1} satisfy $n_1^{2k-1}n_2^{k-1} \mid c(n)$.
\end{theorem}

\begin{proof}
	By writing $F$ as a linear combination of Maass-Poinacar\'e series, we obtain from Proposition~\ref{proposition fourier expansion half-integral} that the $n$-th coefficient of \eqref{eq halfintegral 1} equals
	\begin{align}\label{eq Fourier coefficients}
	(-1)^{\left\lfloor\frac{k}{2}\right\rfloor+1} n^{2k-1}\sum_{m|n} \left(\frac{D}{m}\right) m^{-k} c_{F}\left(\frac{n^{2}|D|}{m^{2}}\right).
	\end{align}
	Since the coefficients of $F$ of negative index are integral by assumption, all of its coefficients are integral by Lemma~\ref{lemma integral basis}. Hence we see that \eqref{eq Fourier coefficients} is divisible by $n^{k-1}$. Moreover, if we write $n = n_1 n_2$ where every prime divisor of $n_1$ divides $D$ and $n_2$ is coprime to $D$, then the summand of index $m$ in \eqref{eq Fourier coefficients} vanishes unless $m$ divides $n_2$. Hence $n_2^k$ times the divisor sum is an integer. This proves Theorem~\ref{theorem halfintegral}.
\end{proof}

\begin{proof}[Proof of Theorem~\ref{theorem hecke translates}]
	Let $(-1)^k d > 0$. We show that Theorem~\ref{theorem hecke translates} follows from Theorem~\ref{theorem halfintegral} as a special case. More precisely, we show that
	\begin{align}\label{eq 1}
	f_{k,d,D,\underline{\lambda}} = \sum_{(-1)^k \delta > 0}c_{F_d}(-|\delta|)f_{k,\delta,D},
	\end{align}
	where 
	\begin{align}\label{eq F}
	F_d := \sum_{m \geq 1}\lambda_m P_{\frac{3}{2}-k,|d|}|T_{m^2}, 
	\end{align}
	with the Hecke operator $T_{m^2}$ on the space $H_{\frac{3}{2}-k}$ of harmonic Maass forms of weight $\frac{3}{2}-k$ defined in Section~\ref{section hecke operators}. Furthermore, we show that $F$ is weakly holomorphic. 
	
	Let $\Phi_D$ be the linear map sending the Maass-Poincar\'e series $P_{\frac{3}{2}-k,|d|}$ to $f_{k,d,D}$. Using the Hecke actions \eqref{eq hecke action fkd} and \eqref{eq hecke action poincare series}, and the fact that $T_m$ and $T_{m^2}$ are defined by the compatible recursions \eqref{eq hecke recursion integral weight}, \eqref{eq hecke recursion half-integral weight}, we obtain 
	\begin{align*}
	f_{k,d,D,\underline{\lambda}} &= \sum_{m=1}^\infty \lambda_m\Phi_D\left(P_{\frac{3}{2}-k,|d|}\right)|T_{m} \\
	&=\Phi_D\left(\sum_{m=1}^\infty \lambda_m P_{\frac{3}{2}-k,|d|}|T_{m^2}\right) = \Phi_D(F_d) \\
	&= \sum_{(-1)^k\delta > 0}c_{F_d}(-|\delta|)\Phi_D\left(P_{\frac{3}{2}-k,|\delta|}\right) = \sum_{(-1)^k\delta > 0}c_{F_d}(-|\delta|)f_{k,\delta,D},
	\end{align*}
	which yields \eqref{eq 1}. 
	
	It remains to show that the function $F_d$ in \eqref{eq F} is weakly holomorphic, which is equivalent to $\xi_{\frac{3}{2}-k}F_d = 0$. Using $\xi_{\frac{3}{2}-k}(F|T_{m^2}) = (\xi_{\frac{3}{2}-k}(F))|T_{m^2}$ for any $F \in H_{\frac{3}{2}-k}$, which can be checked directly for $T_{p^2}$ on the Fourier expansions, we obtain
	\begin{align}\label{eq xi}
	\xi_{\frac{3}{2}-k}\left(\sum_{m=1}^\infty \lambda_m P_{\frac{3}{2}-k,|d|}|T_{m^2}\right) = \sum_{m=1}^\infty \lambda_m \left(\xi_{\frac{3}{2}-k}P_{\frac{3}{2}-k,|d|}\right)\big|T_{m^2}.
	\end{align}
	Recall from \cite{kohnenhalfintegral} that the Shimura correspondence $\mathcal{S}: S_{\frac{1}{2}+k} \to S_{2k}$ is a Hecke-equivariant isomorphism, that is, we have $\mathcal{S}(F)|T_{m} = \mathcal{S}(F|T_{m^2})$. Hence the function on the right-hand side of \eqref{eq xi} corresponds to
	\[
	\sum_{m=1}^\infty \lambda_m \mathcal{S}\left(\xi_{\frac{3}{2}-k}P_{\frac{3}{2}-k,|d|}\right)\big|T_m \in S_{2k}
	\]
	under the Shimura correspondence. But this sum vanishes since $(\lambda_m)$ is a relation for $S_{2k}$, see \cite{loebrichschwagenscheidt}, Lemma 6.1. Hence also the right-hand side of \eqref{eq xi} vanishes. This shows that the function $F_d$ in \eqref{eq F} is weakly holomorphic. Hence Theorem~\ref{theorem hecke translates} follows from Theorem~\ref{theorem halfintegral} in the case $(-1)^k d > 0$.
	
	Now let $(-1)^k d < 0$. Note that the right-hand side in Lemma~\ref{lemma case change} is an integral linear combination of functions $f_{k,\delta,\Delta}$ with fundamental $(-1)^k \delta > 0$. In particular, after taking Hecke translates, we can conclude the claimed properties of $f_{k,\ell^2 d_0, D,\underline{\lambda}}$ from the properties of $f_{k,a^2 D,d_0, \underline{\lambda}}$ shown above.
\end{proof}

\section{The proof of Proposition~\ref{proposition sign changes}}
\label{section proof sign changes}

In order to prove Proposition~\ref{proposition sign changes} we use the Fourier expansion of $f_{k,d,D}$ from Proposition~\ref{proposition fourier expansion bessel}. We need the following estimates for the Sali\'e sums and Bessel functions appearing in this expansion.

\begin{lemma}\label{lemma salie estimate}
	For every $a,n \in \N$ and every $d,D \in \Z$ with $dD \equiv 0,1 \pmod 4$ we have
	\[
	|S_{a,d,D}(n)| \leq \frac{2\sqrt{2}}{\sqrt{3}}\sqrt{a}.
	\]
\end{lemma}

\begin{proof}
	Estimating every summand by $1$, we have
	\[
	|S_{a,d,D}(n)| \leq \#\{b \in \Z/2a\Z: b^2 \equiv dD \pmod{4a}\} =: r_{dD}^*(a).
	\]
	An explicit formula for $r_{\Delta}^*(a)$ can be deduced from the proof of Proposition 2 in Chapter 1.2 in \cite{hirzebruchzagier}. If we write $\Delta = \Delta_0 f^2$ with a fundamental discriminant $\Delta_0$ and $f \in \N$, then we have
	\[
	r_{\Delta}^*(a) = \prod_{p \mid a}\alpha_p(a), \quad
	\alpha_p(a) = \begin{dcases}
	p^{\left\lfloor \frac{\nu_p(a)}{2} \right\rfloor}, & \text{if } \nu_p(a) \leq 2\nu_p(f), \\
	p^{\nu_p(f)}\left( \frac{\Delta_0}{p}\right)^{\nu_p(a)-2\nu_p(f)-1}\left[1+\left( \frac{\Delta_0}{p}\right) \right], & \text{if } \nu_p(a) > 2\nu_p(f).
	\end{dcases}
	\]
	For $p\geq 5$ we have $\left|1+\left( \frac{\Delta_0}{p}\right)\right| \leq 2 \leq \sqrt{p}$. Hence we can estimate
	\begin{align*}
	|\alpha_2(a)| &\leq \sqrt{2}\sqrt{2^{\nu_2(n)}}, \\
	|\alpha_3(a)| &\leq \frac{2}{\sqrt{3}}\sqrt{3^{\nu_3(n)}}, \\
	|\alpha_p(a)| &\leq \sqrt{p^{\nu_p(n)}}, \quad \text{if } p \geq 5.
	\end{align*}
	This yields the claimed estimate.
\end{proof}

We remark that we actually have $|S_{a,d,D}(n)| \ll_\varepsilon a^{\varepsilon}$ for every $\varepsilon > 0$. However, we need an explicit constant in our estimate.

\begin{lemma}\label{lemma bessel estimates}
	For $\nu \in\C, x > 0$ and $a \geq 1$ we have
	\[
	I_{\nu}\left(\frac{x}{a}\right) \leq a^{-\nu}I_{\nu}(x).
	\]
	Moreover, for $\nu = \frac{3}{2}$, $x \geq \sqrt{3}\pi$, $a \geq 2$, we have
	\[
	I_{\frac{3}{2}}\left(\frac{x}{a}\right) \leq \frac{1}{4} a^{-\frac{3}{2}}I_{\frac{3}{2}}(x).
	\]
\end{lemma}

\begin{proof}
	For fixed $x > 0$ we consider the function
	\[
	\phi(a) := \frac{I_{\nu}\left(\frac{x}{a}\right)}{a^{-\nu}I_{\nu}(x)}.
	\]
	We have $\phi(1) = 1$. We show that $\phi(a)$ it is monotonically decreasing, which will prove the first part of the lemma. Using $I_{\nu}'(z) = I_{\nu+1}(z)+\frac{\nu}{z}I_{\nu}(z)$ (see (9.1.21) in \cite{abramowitz}), we compute
	\begin{align*}
	\frac{d}{da}a^\nu I_{\nu}\left(\frac{x}{a}\right)& = \nu a^{\nu-1}I_{\nu}\left(\frac{x}{a}\right)-a^{\nu-2} x I_{\nu}'\left(\frac{x}{a}\right) \\
	&= \nu a^{\nu-1}I_{\nu}\left(\frac{x}{a}\right)-a^{\nu-2} x \left( I_{\nu+1}\left(\frac{x}{a}\right)+\frac{\nu a}{x}I_{\nu}\left(\frac{x}{a}\right)\right) = -a^{\nu-2}xI_{\nu+1}\left( \frac{x}{a}\right).
	\end{align*}
	Since $a$ and $x$ are positive, and the $I$-Bessel function is positive on positive arguments, the derivative of $\phi(a)$ is negative, which shows that $\phi(a)$ is decreasing.
	
	In order to prove the second part of the lemma, we first note that
	\[
	\frac{I_{\frac{3}{2}}\left(\frac{x}{a}\right)}{a^{-\frac{3}{2}}I_{\frac{3}{2}}(x)} \leq \frac{I_{\frac{3}{2}}\left(\frac{x}{2}\right)}{2^{-\frac{3}{2}}I_{\frac{3}{2}}(x)}
	\]
	for $a \geq 2$ and every $x > 0$, since the quotient on the left-hand side is decreasing as a function of $a$. Note that 
	\[
	I_{\frac{3}{2}}(x) = \frac{2x\cosh(x)-2\sinh(x)}{x\sqrt{2\pi x}}.
	\]	
	Now one can show that $I_{\frac{3}{2}}\left(\frac{x}{2}\right)/(2^{-\frac{3}{2}}I_{\frac{3}{2}}(x))$ is monotonically decreasing for $x > 0$ using elementary calculus. Hence we obtain for $x \geq \sqrt{3}\pi$ that
	\[
	\frac{I_{\frac{3}{2}}\left(\frac{x}{2}\right)}{2^{-\frac{3}{2}}I_{\frac{3}{2}}(x)} \leq \frac{I_{\frac{3}{2}}\left(\frac{\sqrt{3}\pi}{2}\right)}{2^{-\frac{3}{2}}I_{\frac{3}{2}}(\sqrt{3}\pi)} \sim 0.206 < \frac{1}{4},
	\]
	which yields the desired estimate.
\end{proof}

\begin{proof}[Proof of Proposition~\ref{proposition sign changes}]
	By Proposition~\ref{proposition fourier expansion bessel} the sign of $c_{f_{k,d,D}}(n)$ equals $(-1)^k$ times the sign of
	\[
	\sum_{a \geq 1}a^{-\frac{1}{2}}S_{a,d,D}(n)I_{k-\frac{1}{2}}\left(\frac{\pi n \sqrt{|dD|}}{a} \right).
	\]
	We first note that for $a = 1$ the Sali\'e sum equals $S_{1,d,D}(n)=(-1)^{ndD}$. Hence it suffices to show that the remaining series satisfies
	\[
	\left|\sum_{a \geq 2}a^{-\frac{1}{2}}S_{a,d,D}(n)I_{k-\frac{1}{2}}\left(\frac{\pi n \sqrt{|dD|}}{a} \right) \right| < I_{k-\frac{1}{2}}\left(\pi n \sqrt{|dD|} \right).
	\]
	Let us first suppose that $k \geq 3$. We use the estimate
	\[
	|S_{a,d,D}(n)| \leq \frac{2\sqrt{2}}{\sqrt{3}}\sqrt{a}
	\]
	from Lemma~\ref{lemma salie estimate} and the estimate
	\[
	I_{k-\frac{1}{2}}\left(\frac{\pi n \sqrt{|dD|}}{a} \right) \leq a^{\frac{1}{2}-k}I_{k-\frac{1}{2}}\left(\pi n \sqrt{|dD|} \right)
	\]
	from Lemma~\ref{lemma bessel estimates} to compute
	\begin{align*}
	\left|\sum_{a \geq 2}a^{-\frac{1}{2}}S_{a,d,D}(n)I_{k-\frac{1}{2}}\left(\frac{\pi n \sqrt{|dD|}}{a} \right) \right| & \leq \frac{2\sqrt{2}}{\sqrt{3}}\sum_{a\geq 2}a^{-\frac{1}{2}}\cdot a^{\frac{1}{2}}\cdot a^{\frac{1}{2}-k}I_{k-\frac{1}{2}}\left(\pi n \sqrt{|dD|} \right) \\
	&= \frac{2\sqrt{2}}{\sqrt{3}}\left(\zeta\left(k-\frac{1}{2}\right)-1\right)I_{k-\frac{1}{2}}\left(\pi n \sqrt{|dD|} \right) \\
	&\leq \frac{2\sqrt{2}}{\sqrt{3}}\left(\zeta\left(\frac{5}{2}\right)-1\right)I_{k-\frac{1}{2}}\left(\pi n \sqrt{|dD|} \right)\\
	&< 0.6\cdot I_{k-\frac{1}{2}}\left(\pi n \sqrt{|dD|} \right).
	\end{align*}
	For $k = 2$ we note that $\pi n \sqrt{|dD|} \geq \sqrt{3}\pi$. Hence we can use the second estimate
	\[
	I_{\frac{3}{2}}\left(\frac{x}{a}\right) \leq \frac{1}{4} a^{-\frac{3}{2}}I_{\frac{3}{2}}(x)
	\]
	for $x \geq \sqrt{3}\pi$ from Lemma~\ref{lemma bessel estimates} to compute
	\begin{align*}
	\left|\sum_{a \geq 2}a^{-\frac{1}{2}}S_{a,d,D}(n)I_{\frac{3}{2}}\left(\frac{\pi n \sqrt{|dD|}}{a} \right) \right| & \leq \frac{1}{\sqrt{6}}\sum_{a\geq 2}a^{-\frac{1}{2}}\cdot a^{\frac{1}{2}}\cdot a^{-\frac{3}{2}}I_{\frac{3}{2}}\left(\pi n \sqrt{|dD|} \right) \\
	&\leq \frac{1}{\sqrt{6}}\left(\zeta\left(\frac{3}{2}\right)-1\right)I_{\frac{3}{2}}\left(\pi n \sqrt{|dD|} \right)\\
	&< 0.7\cdot I_{\frac{3}{2}}\left(\pi n \sqrt{|dD|} \right).
	\end{align*}
	This finishes the proof.
\end{proof}

\section{The proof of Proposition~\ref{proposition combinatorial interpretation}}
\label{section proof combinatorial interpretation}

Let $N \in \N$. In order to prove Proposition~\ref{proposition combinatorial interpretation} we will write the coefficients of the meromorphic modular forms $f_{k,d,r,D,\rho}$ for $\Gamma_0(N)$ in terms of the coefficients of Maass-Poincar\'e series of half-integral weight, similarly as in Proposition~\ref{proposition fourier expansion half-integral}. However, it is now more convenient to work with vector-valued modular forms for the Weil representation. We briefly describe the setup, following \cite{bruinierhabil}.

We consider the lattice $L = \Z$ together with the quadratic form $Q(x) = Nx^2$. The corresponding bilinear form is given by $(x,y) = 2Nxy$. The dual lattice of $L$, which consists of those $x \in \Q$ satisfying $(x,y) \in \Z$ for all $y \in L$, is given by $L' = \frac{1}{2N}\Z$. Hence we can identify the discriminant group $L'/L$ with $\Z/2N\Z$. We let $\e_r$ with $r \in \Z/2N\Z$ denote the standard basis of the group ring $\C[\Z/2N\Z]$. Let $\Mp_2(\Z)$ be the metaplectic double cover of $\SL_2(\Z)$, which can be realized as the set of all pairs $(M,\phi)$ with $M = \left(\begin{smallmatrix} a & b \\ c & d \end{smallmatrix} \right) \in \SL_2(\Z)$ and $\phi: \H \to \C$ holomorphic with $\phi^2(\tau) = c\tau + d$. The Weil representation $\rho_N$ of $\Mp_2(\Z)$ on $\C[\Z/2N\Z]$ is defined on the generators $T = \left(\left(\begin{smallmatrix}1 & 1  \\ 0 & 1 \end{smallmatrix}\right), 1 \right)$ and $S = \left(\left(\begin{smallmatrix}0 & -1  \\ 1 & 0 \end{smallmatrix}\right), \sqrt{\tau} \right)$ of $\Mp_2(\Z)$ by the formulas
\[
\rho_N(T)\e_r = e(Q(r))\e_r, \qquad \rho_N(S)\e_r = \frac{e(-1/8)}{\sqrt{2N}}\sum_{r'(2N)}e(-(r,r'))\e_{r'}.
\]
The dual Weil representation $\rho_N^*$ is obtained by taking complex conjugates in the above formulas everywhere possible.

Let $\kappa \in \Z$. We say that a function $F: \H \to \C[\Z/2N\Z]$ transforms like a modular form of weight $\frac{1}{2}+\kappa$ for $\rho_N$ if it satisfies
\[
F(M\tau) = \phi(\tau)^{1+2\kappa}\rho_N(M,\phi)F(\tau)
\]
for every $(M,\phi) \in \Mp_2(\Z)$. We call $F$ a harmonic Maass form of weight $\frac{1}{2}+\kappa$ for $\rho_N$ if it transforms like a modular form of weight $\frac{1}{2}+\kappa$, is annihilated by the Laplace operator $\Delta_{\frac{1}{2}+\kappa}$ defined in \eqref{eq laplace}, and has a Fourier expansion of the shape
\[
F(\tau) = \sum_{r (2N)}\sum_{\substack{n \equiv r^2(4N) \\ n\gg -\infty}}c_F^+(n,r)q^{\frac{n}{4N}}\e_r + \sum_{r (2N)}\sum_{\substack{n \equiv r^2(4N) \\ n< 0}}c_F^-(n,r)\Gamma\left(\frac{1}{2}-\kappa,\frac{\pi |n| v}{N}\right)q^{\frac{n}{4N}}\e_r.
\]
If the non-holomorphic part of $F$ vanishes, that is, $c_F^-(n,r) = 0$ for all $n < 0$ and $r \in \Z/2N\Z$, then $F$ is a weakly holomorphic modular form. Harmonic Maass forms for the dual Weil representation $\rho_N^*$ are defined analogously, but one has to replace the condition $n \equiv r^2(4N)$ in the Fourier expansion by $-n \equiv r^2 (4N)$.

For every $r \in \Z/2N\Z$ and $n > 0$ with $(-1)^{\kappa+1} n \equiv r^2 \pmod{4N}$, one can construct a Maass-Poincar\'e series 
\[
P_{\frac{1}{2}+\kappa,n,r}(\tau) = q^{-\frac{n}{4N}}(\e_r+\e_{-r})+O(1)
\]
of weight $\kappa$ for $\rho_N$ (if $\kappa$ is even) or $\rho_N^*$ (if $\kappa$ is odd). If $\kappa \geq 2$ then $P_{\frac{1}{2}+\kappa,n,r}$ is weakly holomorphic. We refer to \cite{bruinierhabil}, Section 1.3, for more details on these Poincar\'e series.

\begin{proposition}\label{proposition fourier expansion half-integral vector-valued}
	For $k \geq 2$, and $n \geq 1$ we have
	\begin{align*}
	&c_{f_{k,d,r,D,\rho}}(n)+c_{f_{k,d,-r,D,\rho}}(n) \\
	 &\quad= 
	\begin{dcases} 
	(-1)^{\left\lfloor\frac{k}{2}\right\rfloor+1} n^{2k-1}\sum_{m|n} \left(\frac{D}{m}\right) m^{-k} c_{P_{\frac{3}{2}-k,|d|,r}}^+\left(\frac{n^{2}|D|}{m^{2}},\frac{n\rho}{m}\right), &  \text{if } (-1)^k d > 0, \\
	(-1)^{\left\lfloor\frac{k}{2}\right\rfloor} \ell^{2k-1}\sum_{m|n} \left(\frac{D}{m}\right) m^{k-1} c_{P_{\frac{1}{2}+k,|d|,r}}\left(\frac{n^{2}|D|}{m^{2}},\frac{n\rho}{m}\right) ,& \text{if } (-1)^k d < 0,
	\end{dcases}
	\end{align*}
	where we write $d = \ell^2 d_0$ with a fundamental discriminant $d_0$ and $\ell \in \N$, and $c_{P_{\frac{3}{2}-k,|d|,r}}^+(n)$ denotes the $n$-th coefficient of the holomorphic part of the harmonic Maass form $P_{\frac{3}{2}-k,|d|,r}$.
\end{proposition}

\begin{proof}
	The idea of the proof is similar to the proof of Proposition~\ref{proposition fourier expansion half-integral}, but we need higher-level versions of the results cited there. Putting together the results from the literature is straightforward, but lengthy, so we will only give a sketch. First, the Fourier expansion of $f_{k,d,r,D,\rho}$ can be computed as in \cite{bengoecheapaper}, Proposition 2.2. The resulting expansion looks like the one given in Proposition~\ref{proposition fourier expansion bessel} above, but one has to add the condition $N \mid a$ into the sum over $a \geq 1$, and the condition $b \equiv r\rho \pmod{2N}$ into the Sali\'e sum. The Fourier expansion of the vector-valued Maass-Poincar\'e series $P_{\frac{1}{2}+\kappa,n,r}$ in terms of Kloosterman sums and the $I$-Bessel function is computed in \cite{bruinierhabil}, Proposition 1.10. Using Shintani's formula for the Weil representation given in \cite{bruinierhabil}, Proposition 1.1, it is easy to see that the Kloosterman sums appearing in the Fourier coefficients of the vector-valued Maass-Poincar\'e series agree with the Kloosterman sums appearing in the Fourier coefficients of the Jacobi-Poincar\'e series given in the proposition on p. 519 in \cite{grosskohnenzagier}. Finally, the Sali\'e sum appearing in the Fourier expansion of $f_{k,d,r,D,\rho}$ can be written as a divisor sum involving these Jacobi-Kloosterman sums by the lemma on p. 524 in \cite{grosskohnenzagier}, which then gives the formula above.
\end{proof}

\begin{proof}[Proof of Proposition~\ref{proposition combinatorial interpretation}]
	The generating function of the partition function $p(n)$ is given by the inverse of the Dedekind eta function, that is,
	\[
	\eta^{-1}(\tau) = q^{-\frac{1}{24}}\sum_{n \geq 0}p(n)q^n.
	\]
	The transformation behaviour of $\eta^{-1}$ under $\SL_2(\Z)$ implies that the $\C[\Z/12\Z]$-valued function
	\[
	\sum_{r(12)}\left( \frac{12}{r}\right)\eta^{-1}(\tau)\e_r = (\e_1 - \e_5 - \e_7 +\e_{11})\eta^{-1}(\tau)
	\]
	is a weakly holomorphic modular form of weight $-\frac{1}{2}$ for the dual Weil representation $\rho_6^*$, see \cite{bruinierono}, Section 3.2. Its principal part is given by $(\e_1 - \e_5 - \e_7 + \e_{11})q^{-\frac{1}{24}}$, which implies that it can be written as a linear combination of Maass-Poincar\'e series as
	\[
	\sum_{r(12)}\left( \frac{12}{r}\right)\eta^{-1}(\tau)\e_r = P_{-\frac{1}{2},1,1}(\tau) - P_{-\frac{1}{2},1,5}(\tau).
	\]
	From Proposition~\ref{proposition fourier expansion half-integral vector-valued} we obtain (with $k = 2$, $D = -23, \rho = 1, d = 1$, and $r = 1,5$) that the $n$-th coefficient of 
	\begin{align}\label{eq Fminus23 alternative}
	f_{2,1,1,-23,1}+f_{2,1,11,-23,1}-f_{2,1,5,-23,1}-f_{2,1,7,-23,1}
	\end{align}
	is given by
	\[
	n^{3}\sum_{m|n} \left(\frac{-23}{m}\right)\left( \frac{12}{n/m}\right) m^{-2}p\left(\frac{23n^{2}}{24m^{2}}+\frac{1}{24}\right).
	\]
	Note that we have $f_{k,d,r,D,\rho} = f_{k,D,r,d,\rho}$ if $d$ and $D$ are both fundamental with $d \equiv D \pmod{4N}$, since $\chi_D = \chi_d$ on primitive forms of discriminant $dD$. Hence the linear combination \eqref{eq Fminus23 alternative} equals the function $F_{-23}$ defined in \eqref{eq Fminus23} and we obtain the first formula in Proposition~\ref{proposition combinatorial interpretation}.
	
	The second formula follows in a similar way. This time we consider the function
	\[
	\sum_{r(12)}\left( \frac{12}{r}\right)\eta(\tau)j'(\tau)\e_r,
	\]
	which is a weakly holomorphic modular form of weight $\frac{5}{2}$ for the Weil representation $\rho_6$. It has principal part $-(\e_1 - \e_5 - \e_7 + \e_{11})q^{-\frac{23}{24}}$ and can be written in terms of Maass-Poincar\'e series as
	\[
	\sum_{r(12)}\left( \frac{12}{r}\right)\eta(\tau)j'(\tau)\e_r = -P_{\frac{5}{2},23,1}(\tau)+P_{\frac{5}{2},23,5}(\tau).
	\]
	Here one also has to compare the first few Fourier coefficients since the difference of both sides may be a non-zero holomorphic modular form. Now the Fourier expansion $\eta(\tau) = \sum_{n \geq 1}\left( \frac{12}{n}\right)q^{\frac{n^2}{24}}$ implies 
	\[
	\sum_{r(12)}\left( \frac{12}{r}\right)\eta(\tau)j'(\tau)\e_r = \sum_{r(12)}\left(\frac{12}{r}\right)\sum_{\substack{\ell \geq -23 \\ \ell \equiv 1 (12)}}\sum_{d = 1}^{\infty}\left( \frac{12}{d}\right)c_{j'}\left( \frac{\ell-d^2}{24}\right)q^{\frac{\ell}{24}}\e_r.
	\]
	Finally, applying Proposition~\ref{proposition fourier expansion half-integral vector-valued} (with $k = 2, D = 1, \rho = 1, d = -23$ and $r = 1,5$) yields the second formula in Proposition~\ref{proposition combinatorial interpretation}. 
\end{proof}


\begin{thebibliography}{99}
	\bibitem{aas}H.-F. Aas, {\it Congruences for the coefficients of the modular invariant $j(\tau)$}, Math. Scand. {\bf 14} (1964), 185--192.
	\bibitem{abramowitz} M.~Abramowitz and I.~A.~Stegun, {\it Handbook of mathematical functions with formulas, graphs, and mathematical tables}, National Bureau of Standards Applied Mathematics Series {\bf 55} (1964).
	\bibitem{bengoecheapaper} P.~Bengoechea, {\it Meromorphic analogues of modular forms generating the kernel of {S}hintani's lift}, Math. Res. Lett., {\bf 22} (2015), 337--352.
	\bibitem{broadhurstzudilin} D. Broadhurst and W. Zudilin, {\it A magnetic double integral}, J. Aust. Math. Soc. {\bf 107} (2019), no. 1, 9--25.
	\bibitem{borcherds} R.E.~Borcherds, {\it Automorphic forms with singularities on {G}rassmannians}, Invent. Math. {\bf 132} (1998), 491--562.
	\bibitem{bringmannkanevonpippich} K.~Bringmann, B.~Kane, A.~ von Pippich, {\it Regularized inner products of meromorphic modular forms and higher {G}reen's Functions}, Commun. Contemp. Math. (2018), doi:10.1142/S0219199718500293.
	\bibitem{bringmannono} K. Bringmann and K. Ono, {\it Arithmetic properties of coefficients of half-integral weight Maass-Poincar\'e series},
Math. Ann. {\bf 337} (2007), 591--612.
	\bibitem{bruinierhabil} J.H.~Bruinier, {\it Borcherds products on {O}(2, {$l$}) and {C}hern classes of {H}eegner divisors}, Lecture Notes in Mathematics {\bf 1780}, Springer-Verlag, Berlin (2002).
	\bibitem{bruinierfunke} J.H. Bruinier and J. Funke, {\it On two geometric theta lifts}, Duke Math. J. {\bf 125} (2004), no. 1, 45--90. 
	\bibitem{bruinierono} J.H.~Bruinier and K. Ono, {\it Algebraic formulas for the coefficients of half-integral weight harmonic weak {M}aass forms}, Adv. Math. {\bf 246} (2013), 198--219.
	\bibitem{dukeimamoglutoth} W.~Duke, \"O.~Imamo\={g}lu, and \'A.~T\'oth, {\it Cycle integrals of the $j$-function and mock modular forms}, Ann. of Math. (2) {\bf 173} (2011), no. 2, 947-981.
	\bibitem{dukejenkins} W.~Duke and P.~Jenkins, {\it Integral traces of singular values of weak Maass forms}, Algebra Number Theory {\bf2} (2008), no. 5, 573--593.
	\bibitem{dukejenkins2} W.~Duke and P.~Jenkins, {\it On the zeros and coefficients of certain weakly holomorphic modular forms}, Pure Appl. Math. Q. {4} (2008), no. 4, Special Issue: In honor of Jean-Pierre Serre. Part 1, 1327--1340.
	\bibitem{griffin} M. Griffin, {\it Divisibility properties of coefficients of weight 0 weakly holomorphic modular forms}, Int. J.
Number Theory {\bf 7} (2011), no. 4, 933--941.
	\bibitem{grosskohnenzagier} B. Gross, W. Kohnen, D. Zagier, {\it Heegner points and derivatives of $L$-series. II}, Math. Ann. {\bf 278} (1987), 497--562.
	\bibitem{grosszagier} B.~Gross and D.~Zagier, {\it Heegner points and derivatives of $L$-series}, Invent. Math. {\bf 84}, 225-320 (1986).
	\bibitem{hirzebruchzagier} F.~Hirzebruch and D.~Zagier, {\it Intersection numbers of curves on Hilbert modular surfaces and modular forms of nebentypus}, Invent. Math. {\bf 36} (1976), 57--113.
	\bibitem{hondakaneko} Y. Honda and M. Kaneko, {\it On Fourier coefficients of some meromorphic modular forms}, Bull. Korean Math. Soc. {\bf 49} (2012), no. 6, 1349--1357.
	\bibitem{jenkinsmolnar} P. Jenkins and G. Molnar, {\it Zagier duality for level $p$ weakly holomorphic modular forms}, Ramanujan J. {\bf 50} (2019), 93--109.		\bibitem{kohlberg} O. Kohlberg, {\it Congruences for the coefficients of the modular invariant $j(\tau)$}, Math. Scand. {\bf 10} (1962), 173--
181.
	\bibitem{kohnenhalfintegral} W. Kohnen, {\it Modular forms of half-integral weight on $\Gamma_0(4)$}, Math. Ann. {\bf 248} (1980), 249--266.
 	\bibitem{lehner1} J.~Lehner, {\it Divisibility properties of the Fourier coefficients of the modular invariant $j(\tau)$}, Amer. J. Math.
{\bf 71} (1949), 136--148.
	\bibitem{lehner2} J.~Lehner, {\it Further congruence properties of the Fourier coefficients of the modular invariant $j(\tau)$}, Amer.
J. Math. {\bf 71} (1949), 373--386.
	\bibitem{lineururer} Y. Li and M. Neururer, {\it A magnetic modular form}, Int. J. Number Theory {\bf 15} (2019), no. 5, 907--924. 
	\bibitem{loebrichschwagenscheidt} S. L\"obrich and M. Schwagenscheidt, {\it Meromorphic modular forms with rational cycle integrals}, Int. Math. Res. Not. (2020), 1--31.
	\bibitem{pasolzudilin} V. Pa\c{s}ol and W. Zudilin, {\it Magnetic (quasi-)modular forms}, preprint arxiv:2009.14609 (2020).
	\bibitem{zagiereisenstein} D. Zagier, {\it Eisenstein series and the Riemann zeta function}, in: Automorphic Forms, Representation Theory and Arithmetic, Springer-Verlag, Berlin-Heidelberg-New York (1981), 275--301.
\end{thebibliography}
\end{document}